\pdfoutput=1
\RequirePackage{ifpdf}
\ifpdf 
\documentclass[pdftex]{sigma}
\else
\documentclass{sigma}
\fi

\numberwithin{equation}{section}

\newtheorem{Theorem}{Theorem}[section]
\newtheorem{Proposition}[Theorem]{Proposition}
 { \theoremstyle{definition}

\newtheorem{Remark}[Theorem]{Remark} }

\begin{document}

\newcommand{\arXivNumber}{2003.13512}

\renewcommand{\PaperNumber}{014}

\FirstPageHeading

\ShortArticleName{The Subelliptic Heat Kernel of the Octonionic Anti-De Sitter Fibration}

\ArticleName{The Subelliptic Heat Kernel\\ of the Octonionic Anti-De Sitter Fibration}

\Author{Fabrice BAUDOIN and Gunhee CHO}

\AuthorNameForHeading{F.~Baudoin and G.~Cho}

\Address{Department of Mathematics, University of Connecticut,\\ 196 Auditorium Road, Storrs, CT 06269-3009, USA}
\Email{\href{mailto:fabrice.baudoin@uconn.edu}{fabrice.baudoin@uconn.edu}, \href{mailto:gunhee.cho@uconn.edu}{gunhee.cho@uconn.edu}}

\ArticleDates{Received July 31, 2020, in final form January 29, 2021; Published online February 10, 2021}

\Abstract{In this note, we study the sub-Laplacian of the $15$-dimensional octonionic anti-de Sitter space which is obtained by lifting with respect to the anti-de Sitter fibration the Laplacian of the octonionic hyperbolic space $\mathbb{O}H^1$. We also obtain two integral representations for the corresponding subelliptic heat kernel.}

\Keywords{sub-Laplacian; 15-dimensional octonionic anti-de Sitter space; the anti-de Sitter fibration}

\Classification{58J35; 53C17}

\section{Introduction and results}

In this note we study the sub-Laplacian and the corresponding sub-Riemannian heat kernel of the octonionic anti-de Sitter fibration
\begin{equation*}
\mathbb{S}^7 \hookrightarrow {\rm AdS}^{15}(\mathbb{O}) \rightarrow \mathbb{O}H^1.
\end{equation*}

This paper follows the previous works \cite{FBND18,FBJW18, JW16} which respectively concerned:
\begin{enumerate}\itemsep=0pt
\item {\it The complex anti-de Sitter fibrations}:
\[
\mathbb{S}^1 \hookrightarrow {\rm AdS}^{2n+1}(\mathbb{C}) \rightarrow {\mathbb{C}H}^n.
\]
\item {\it The quaternionic anti-de Sitter fibrations}:
\[
\mathbb{S}^3 \hookrightarrow {\rm AdS}^{4n+3}(\mathbb{H}) \rightarrow {\mathbb{H}H}^n.
\]
\end{enumerate}

The 15-dimensional anti-de Sitter fibration is the last model space that remained to be studied of a sub-Riemannian manifold arising from a $H$-type semi-Riemannian submersion over a rank-one symmetric space, see the Table~3 in~\cite{BGMR}.

Similarly to the complex and quaternionic case, the sub-Laplacian is defined as the lift on ${\rm AdS}^{15}(\mathbb{O})$ of the Laplace--Beltrami operator of the octonionic hyperbolic space $\mathbb{O}H^1$. However, in the complex and quaternionic case the Lie group structure of the fiber played an important role that we can not use here, since the fiber $\mathbb{S}^7$ is not a group. Instead, we make use of some algebraic properties of $\mathbb S^7$ that were already pointed out and used by the authors in~\cite{FBGC19} for the study of the octonionic Hopf fibration:
\begin{equation*}
\mathbb{S}^7 \hookrightarrow \mathbb{S}^{15} \rightarrow \mathbb{O}P^1.
\end{equation*}
Let us briefly describe our main results. Due to the cylindrical symmetries of the fibration, the heat kernel of the sub-Laplacian only depends on two variables: the variable $r$ which is the Riemannian distance on $\mathbb{O}H^1$ (the starting point is specified with inhomogeneous coordinate in Section~\ref{section3}) and the variable~$\eta$ which is the Riemannian distance starting at a pole on the fiber~$\mathbb{S}^7$. We prove in Proposition~\ref{prop:1} that in these coordinates, the radial part of the sub-Laplacian $\tilde{L}$ writes
\begin{equation*}
\tilde{L}=\frac{{\partial}^2}{\partial {r}^2}+(7\coth r+7\tanh r)\frac{{\partial}}{\partial {r}}+{\tanh^2 r}\left(\frac{{\partial}^2}{\partial {\eta}^2}+6\cot \eta\frac{{\partial}}{\partial {\eta}}\right).
\end{equation*}

As a consequence of this expression for the sub-Laplacian, we are able to derive two equivalent formulas for the heat kernel. The first formula, see Proposition~\ref{prop:2}, reads as follows: for $r\geq 0$, $\eta\in [0,\pi)$, $t>0$
\begin{align*}
p_t(r,\eta)=\int_{0}^{\infty} s_t(\eta , {\rm i}u) q_{t,15}(\cosh r \cosh u) \sinh^6u \, {\rm d}u,
\end{align*}
where $s_t$ is the heat kernel of the Jacobi operator
\begin{equation*}
\tilde{\triangle}_{\mathbb{S}^7}=\frac{{\partial}^2}{\partial {\eta}^2}+6\cot \eta\frac{{\partial}}{\partial {\eta}}
\end{equation*}
with respect to the measure $\sin^6\eta \, {\rm d}\eta$, and where $q_{t,15}$ is the Riemannian heat kernel on the $15$-dimensional real hyperbolic space $\mathbb{H}^{15}$ given in~\eqref{eq:riem.heat.kernel.on.H^15}.
The second formula, see Proposition~\ref{prop:3}, writes as follows:
\begin{gather*}
p_t(r,\eta)= \int_0^\pi \int_{0}^{\infty} G_t(\eta, \varphi, u) q_{t,9}(\cosh r \cosh u) \sin^5 \varphi \, {\rm d}u \, {\rm d} \varphi ,
\end{gather*}
where $q_{t,9}$ is Riemannian heat kernel on the $9$-dimensional hyperbolic space $\mathbb{H}^{9}$ and $G_t(\eta, \varphi, u)$ is given in~\eqref{function G}.

Similarly to \cite{FBND18,FBJW18,JW16}, it might be expected that explicit integral representations of the heat kernel might be used to study small-time asymptotics, inside and outside of the cut-locus. Integral representations of heat kernels can also be used to obtain sharp heat kernel estimates, see~\cite{MR2541147}. Those applications of the heat kernel representations we obtain will possibly be addressed in a future research project.

\section{The octonionic anti-de Sitter fibration}\label{section2}

Let
\begin{equation*}
\mathbb{O}=\left\{x=\sum_{j=0}^{7}x_j e_j,\, x_j\in \mathbb{R} \right\},
\end{equation*}
be the division algebra of octonions (see~\cite{MR1803808} for explicit representations of this algebra). We recall that the multiplication rules are given by
\begin{gather*}
e_i e_j=e_j \qquad \text{if} \quad i=0,
\\
e_i e_j=e_i \qquad \text{if} \quad j=0,
\\
e_i e_j=-\delta_{ij}e_0+\epsilon_{ijk}e_k \qquad \text{otherwise},
\end{gather*}
where $\delta_{ij}$ is the Kronecker delta and $\epsilon_{ijk}$ is the completely antisymmetric tensor with value $1$ when $ijk = 123, 145, 176, 246, 257, 347, 365$ (also see~\cite{FBGC19}).
The octonionic norm is defined for $x \in \mathbb{O}$ by
\begin{equation*}
||x||^2=\sum_{j=0}^{7}x^2_j.
\end{equation*}

The octonionic anti-de Sitter space ${\rm AdS}^{15}(\mathbb{O})$ is the quadric defined as the pseudo-hyperbolic space by:
\begin{equation*}
{\rm AdS}^{15}(\mathbb{O})=\big\{(x,y)\in\mathbb{O}^2, \, ||(x,y)||^2_{\mathbb{O}}=-1 \big\},
\end{equation*}
where
\begin{equation*}
||(x,y)||^2_{\mathbb{O}}:=||x||^2-||y||^2.
\end{equation*}

In real coordinates we have $x=\sum_{j=0}^{7}x_j e_j$, $y=\sum_{j=0}^{7}y_j e_j$, and the pseudo-norm can be written as
\begin{equation*}
x_0^2+\cdots +x_7^2-y_0^2-\cdots-y_7^2.
\end{equation*}
As such, ${\rm AdS}^{15}(\mathbb{O})$ is embedded in the flat $16$-dimensional space~$\mathbb{R}^{8,8}$ endowed with the Lo\-rentzian real signature $(8,8)$ metric
\begin{equation*}
{\rm d}s^2={\rm d}x_0^2+\dots+{\rm d}x_7^2-{\rm d}y_0^2-\cdots-{\rm d}y_7^2.
\end{equation*}
Consequently, ${\rm AdS}^{15}(\mathbb{O})$ is naturally endowed with a pseudo-Riemannian structure of signature~$(8,7)$.

Let $\mathbb{O}H^1$ denote the octonionic hyperbolic space. The map $\pi\colon {\rm AdS}^{15}(\mathbb{O}) \rightarrow \mathbb{O}H^1$, given by $(x,y) \mapsto [x:y]=y^{-1}x$ is a pseudo-Riemannian submersion with totally geodesic fibers isometric to the seven-dimensional sphere $\mathbb S^7$. Notice that, as a topological manifold, $\mathbb{O}H^1$ can therefore be identified with the unit open ball in $\mathbb{O}$. The pseudo-Riemannian submersion~$\pi$ yields the octonionic anti-de Sitter fibration
\begin{equation*}
\mathbb{S}^7 \hookrightarrow {\rm AdS}^{15}(\mathbb{O}) \rightarrow \mathbb{O}H^1.
\end{equation*}
For further information on semi-Riemannian submersions over rank-one symmetric spaces, we refer to~\cite{MR1877586}.

\section{Cylindrical coordinates and radial part of the sub-Laplacian}\label{section3}

The sub-Laplacian $L$ on ${\rm AdS}^{15}(\mathbb{O})$ we are interested in is the horizontal Laplacian of the Riemannian submersion $\pi\colon {\rm AdS}^{15}(\mathbb{O}) \to \mathbb{O}H^1$, i.e., the horizontal lift of the Laplace--Beltrami operator of $\mathbb{O}H^1$. It can be written as
\begin{equation}\label{eq:sub_Laplacian}
L=\square_{{\rm AdS}^{15}(\mathbb{O})}+\triangle_\mathcal{V},
\end{equation}
where $\square_{{\rm AdS}^{15}(\mathbb{O})}$ is the d'Alembertian, i.e., the Laplace--Beltrami operator of the pseudo-Rie\-man\-nian metric and $\triangle_\mathcal{V}$ is the vertical Laplacian. Since the fibers of $\pi$ are totally geodesic and isometric to $\mathbb{S}^7\subset {\rm AdS}^{15}(\mathbb{O})$, we note that $\square_{{\rm AdS}^{15}(\mathbb{O})}$ and $\triangle_\mathcal{V}$ are commuting operators, and we can identify
\begin{equation}\label{eq:vertical_Laplacian}
\triangle_\mathcal{V}={\triangle}_{\mathbb{S}^{7}}.
\end{equation}

The sub-Laplacian $L$ is associated with a canonical sub-Riemannian structure on ${\rm AdS}^{15}(\mathbb{O})$ which is of $H$-type, see~\cite{BGMR}.

To study $L$, we introduce a set of coordinates that reflect the cylindrical symmetries of the octonionic unit sphere which provides an explicit local trivialization of the octonionic anti-de Sitter fibration. Consider the coordinates $w\in \mathbb{O}H^1$, where $w$ is the inhomogeneous coordinate on $\mathbb{O}H^1$ given by $w=y^{-1} x$, with $x,y\in {\rm AdS}^{15}(\mathbb{O})$. Consider the north pole $p \in \mathbb{S}^7$ and take $Y_1, \dots , Y_7$ to be an orthonormal frame of $T_p \mathbb{S}^7$. Let us denote $\exp_p$ the Riemannian exponential map at $p$ on $\mathbb{S}^7$. Then the cylindrical coordinates we work with are given by
\begin{equation*}
(w,\theta_1, \dots , \theta_7) \mapsto \left(\frac{ \exp_p\big(\sum_{i=1}^{7}\theta_i Y_i\big) w}{\sqrt{1-{\rho}^2}},\frac{ \exp_p\big(\sum_{i=1}^{7}\theta_i Y_i\big)}{\sqrt{1-{\rho}^2}}\right)\in {\rm AdS}^{15}(\mathbb{O}),
\end{equation*}
where $\rho={\|w\|}$ and $\| \theta \|=\sqrt{\theta_1^2+\cdots + \theta_7^2} <\pi$.

A function $f$ on ${\rm AdS}^{15}(\mathbb{O})$ is called radial cylindrical if it only depends on the two coordinates $(\rho,\eta)\in [0,1)\times [0,\pi]$ where $\eta=\sqrt{\sum_{i=1}^{7}\theta^2_i}$. More precisely $f$ is radial cylindrical if there exists a function $g$ so that
\begin{equation*}
f \left(\frac{ \exp_p\big(\sum_{i=1}^{7}\theta_i Y_i\big) w}{\sqrt{1-{\rho}^2}},\frac{ \exp_p\big(\sum_{i=1}^{7}\theta_i Y_i\big)}{\sqrt{1-{\rho}^2}}\right)=g ( \rho, \eta ).
\end{equation*}

We denote by $\mathcal{D}$ the space of smooth and compactly supported functions on $[0,1) \times [0,\pi)$. Then the radial part of $L$ is defined as the operator $\widetilde{L}$ such that for any~$f\in \mathcal{D}$, we have
\begin{equation}\label{intertwining}
L(f \circ \psi)=\big(\widetilde{L}f\big)\circ \psi.
\end{equation}

We now compute $\widetilde{L}$ in cylindrical coordinates.

\begin{Proposition}\label{prop:1}
The radial part of the sub-Laplacian on ${\rm AdS}^{15}(\mathbb{O})$ is given in the coordinates $(r,\eta)$ by the operator
\begin{equation*}
\widetilde{L}=\frac{{\partial}^2}{\partial {r}^2}+(7\coth r+7\tanh r)\frac{{\partial}}{\partial {r}}+{\tanh^2 r}\left(\frac{{\partial}^2}{\partial {\eta}^2}+6\cot \eta\frac{{\partial}}{\partial {\eta}}\right),
\end{equation*}
where $r=\tanh^{-1}\rho$ is the Riemannian distance on~$\mathbb{O}H^1$ from the origin.
\end{Proposition}

\begin{proof}
Note that the radial part of the Laplace--Beltrami operator on the octonionic hyperbolic space $\mathbb{O}H^1$ is
\begin{equation*}
\widetilde{\triangle}_{\mathbb{O}H^1}=\frac{{\partial}^2}{\partial {r}^2}+(7\coth r+7\tanh r)\frac{{\partial}}{\partial {r}},
\end{equation*}
and the radial part of the Laplace--Beltrami operator on $\mathbb{S}^{7}$ is
\begin{equation}\label{eq:LB_on_S^7}
\widetilde{\triangle}_{\mathbb{S}^7}=\frac{{\partial}^2}{\partial {\eta}^2}+6\cot \eta\frac{{\partial}}{\partial {\eta}}.
\end{equation}
Since the octonionic anti-de Sitter fibration defines a totally geodesic submersion with base space $\mathbb{O}H^1$ and fiber $\mathbb{S}^{7}$, the semi-Riemannian metric on ${\rm AdS}^{15}(\mathbb{O})$ is locally given by a warped product between the Riemannian metric of $\mathbb{O}H^1$ and the Riemannian metric on $\mathbb{S}^{7}$. Hence the radial part of the d'Alembertian becomes
\begin{equation}\label{radial 1}
\widetilde{\square}_{{\rm AdS}^{15}(\mathbb{O})}=\frac{{\partial}^2}{\partial {r}^2}+(7\coth r+7\tanh r)\frac{{\partial}}{\partial {r}} +g(r) \left(\frac{{\partial}^2}{\partial {\eta}^2}+6\cot \eta\frac{{\partial}}{\partial {\eta}} \right),
\end{equation}
for some smooth function $g$ to be computed.

On the other hand, from the isometric embedding ${\rm AdS}^{15}(\mathbb{O}) \subset \mathbb{O} \times \mathbb{O}$, the d'Alembertian on ${\rm AdS}^{15}(\mathbb{O})$ is a restriction of the d'Alembertian on $\mathbb{O} \times \mathbb{O} \simeq \mathbb{R}^{8,8}$ in the sense that for a smooth $f\colon {\rm AdS}^{15}(\mathbb{O}) \to \mathbb{R}$
\[
\square_{{\rm AdS}^{15}(\mathbb{O})} f =\square_{\mathbb O \times \mathbb O} f^*_{/ {\rm AdS}^{15}(\mathbb{O})},
\]
where $\square_{\mathbb O \times \mathbb O}=\sum_{i=0}^7 \big( \frac{\partial^2}{\partial x_i^2} - \frac{\partial^2}{\partial y_i^2} \big)$ and for $x,y \in \mathbb O$ such that $\| y\|^2 -\|x\|^2 >0$, $f^*(x,y)= f \big( \frac{x}{\sqrt{\| y \|^2 -\| x \|^2}}, \frac{y}{\sqrt{\| y \|^2 -\| x \|^2}}\big)$. For the specific choice of the function $f(x,y)=y_1$, one easily computes that $ \square_{\mathbb O \times \mathbb O} f^*_{/ {\rm AdS}^{15}(\mathbb{O})}(x,y) = 15y_1$, thus
\[
\square_{{\rm AdS}^{15}(\mathbb{O})} f (x,y)=15y_1.
\]
For the point with coordinates
\[
\left(\frac{ \exp_p\big(\sum_{i=1}^{7}\theta_i Y_i\big) w}{\sqrt{1-{\rho}^2}},\frac{ \exp_p\big(\sum_{i=1}^{7}\theta_i Y_i\big)}{\sqrt{1-{\rho}^2}}\right)\in {\rm AdS}^{15}(\mathbb{O})
\]
one has
\[
y_1=\frac{\cos \eta}{\sqrt{1-{\rho}^2}}= \cosh r \cos \eta.
\]
We therefore deduce that
\[
\widetilde{\square}_{{\rm AdS}^{15}(\mathbb{O})}(\cosh r \cos \eta) =15 \cosh r \cos \eta.
\]
Using the formula~\eqref{radial 1}, after a straightforward computation, this yields $g(r)=-\frac{1}{\cosh^2 r}$ and therefore
\begin{align*}
\widetilde{\square}_{{\rm AdS}^{15}(\mathbb{O})}&=\frac{{\partial}^2}{\partial {r}^2}+(7\coth r+7\tanh r)\frac{{\partial}}{\partial {r}} -\frac{1}{\cosh^2 r} \left(\frac{{\partial}^2}{\partial {\eta}^2}+6\cot \eta\frac{{\partial}}{\partial {\eta}} \right)\nonumber\\
& =\widetilde{\triangle}_{\mathbb{O}H^1}-\frac{1}{\cosh^2 r}\widetilde{\triangle}_{\mathbb{S}^7}.
\end{align*}

Finally, to conclude, one notes that the sub-Laplacian $L$ is given by the difference between the Laplace--Beltrami operator of ${\rm AdS}^{15}(\mathbb{O})$ and the vertical Laplacian. Therefore by~\eqref{eq:sub_Laplacian} and~\eqref{eq:vertical_Laplacian},
\begin{align*}
\widetilde{L} =\tilde{\square}_{{\rm AdS}^{15}(\mathbb{O})}+ \tilde{\triangle}_{\mathbb{S}^{7}}
 = \frac{{\partial}^2}{\partial {r}^2}+(7\coth r+7\tanh r)\frac{{\partial}}{\partial {r}}+{\tanh^2 r}\left(\frac{{\partial}^2}{\partial {\eta}^2}+6\cot \eta\frac{{\partial}}{\partial {\eta}}\right).\tag*{\qed}
\end{align*}\renewcommand{\qed}{}
\end{proof}

\begin{Remark}As a consequence of the previous result, we can check that the Riemannian measure of~${\rm AdS}^{15}(\mathbb{O})$ in the coordinates~$(r,\eta)$, which is the symmetric and invariant measure for~$\tilde{L}$ is given by
\begin{equation}\label{eq:measure}
{\rm d}\overline{\mu}=\frac{{\pi}^7}{90}\sinh^7 r \cosh^7r\sin^6 \eta \, {\rm d}r \,{\rm d}\eta.
\end{equation}
(See also Remark~2 in~\cite{FBGC19}, which corresponds to the case of the octonionic Hopf fibration.)
\end{Remark}

\section{Integral representations of the subelliptic heat kernel}\label{section4}

In this section, we give two integral representations of the subelliptic heat kernel associated with~$\tilde{L}$. We denote by $p_t(r,\eta)$ the heat kernel of~$\tilde{L}$ issued from the point $r=\eta=0$ with respect to the measure \eqref{eq:measure}. We remark that studying the subelliptic heat kernel associated with~$\tilde{L}$ is enough to study the heat kernel of~$L$, because due to \eqref{intertwining} the heat kernel $h_t(w,\theta)$ of~$L$ issued from the point with cylindric coordinates $w=0$, $\theta=0$ is then given by
\[
h_t(w,\theta)=p_t\big( \tanh^{-1} \| w \|, \| \theta \|\big).
\]

\subsection{First integral representation}\label{section4.1}

We denote by $s_t$ the heat kernel of the operator
\begin{equation*}
\tilde{\triangle}_{\mathbb{S}^7}=\frac{{\partial}^2}{\partial {\eta}^2}+6\cot \eta\frac{{\partial}}{\partial {\eta}}
\end{equation*}
with respect to the reference measure $\sin^6\eta \, {\rm d}\eta$. The operator $\tilde{\triangle}_{\mathbb{S}^7}$ belongs to the family of Jacobi diffusion operators which have been extensively studied in the literature, see for instance the appendix in~\cite{MR3719061} and the references therein. In particular, the spectrum of $\tilde{\triangle}_{\mathbb{S}^7}$ is given by
\[
\mathbf{Sp} \big({-} \tilde{\triangle}_{\mathbb{S}^7}\big)= \{ m(m+6),\, m \in \mathbb{N} \},
\]
and the eigenfunction corresponding to the eigenvalue $m(m+6)$ is $P^{5/2,5/2}_m (\cos \eta)$ where $P^{5/2,5/2}_m$ is the Jacobi polynomial
\[
P_m^{5/2,5/2}(x)=\frac{(-1)^m}{2^mm!\big(1-x^2\big)^{5/2}}\frac{{\rm d}^m}{{\rm d}x^m}\big(1-x^2\big)^{5/2+m}.
\]
As a consequence, one has the following spectral decomposition for the heat kernel:
\begin{gather*}
s_t(\eta,u)
 =\frac{1}{{\pi}} \sum_{m=0}^{+\infty} \frac{2^{4m+7} m !(m+5)![(m+3)!]^2}{(2m+6)!(2m+5)!} {\rm e}^{-m(m+6)t}P_m^{5/2,5/2}(\cos \eta )P_m^{5/2,5/2}(\cos u).
\end{gather*}
\begin{Proposition}\label{prop:2}
For $r\geq 0$, $\eta\in [0,\pi]$, and $t >0$ we have
\begin{equation*}
p_t(r,\eta)=\int_{0}^{\infty} s_t(\eta , {\rm i}u) q_{t,15}(\cosh r \cosh u) \sinh^6u \, {\rm d}u,
\end{equation*}
where
\begin{equation}\label{eq:riem.heat.kernel.on.H^15}
q_{t,15}(\cosh s):=\frac{{\rm e}^{-49t}}{(2\pi)^{7}\sqrt{4\pi t}}\left(-\frac{1}{\sinh s}\frac{{\rm d}}{{\rm d}s}\right)^{7}{\rm e}^{-s^2/4t}
\end{equation}
is the Riemannian heat kernel on the $15$-dimensional real hyperbolic space ${\mathbb{H}}^{15}$.
\end{Proposition}
\begin{proof}
Since $\pi\colon {\rm AdS}^{15}(\mathbb{O}) \to \mathbb{O}H^1$
is a (semi-Riemannian) totally geodesic submersion, the operators $\widetilde{{\square}}_{{\rm AdS}^{15}(\mathbb{O})}$ and $\widetilde{\triangle}_{\mathbb{S}^7}$ commute. Thus
\begin{equation*}
{\rm e}^{t\widetilde{L}}={\rm e}^{t(\widetilde{\square}_{{\rm AdS}^{15}(\mathbb{O})}+\widetilde{\triangle}_{\mathbb{S}^7})}={\rm e}^{t\widetilde{\triangle}_{\mathbb{S}^7}}{\rm e}^{t\widetilde{\square}_{{\rm AdS}^{15}(\mathbb{O})}}.
\end{equation*}
We deduce that the heat kernel of $\tilde{L}$ can be written as
\begin{equation}\label{eq:9}
p_t(r,\eta)=\int_{0}^{\pi}s_t(\eta,u)p^{\widetilde{\square}_{{\rm AdS}^{15}(\mathbb{O})}}_t(r,u)\sin^6u\,{\rm d}u,
\end{equation}
where $s_t$ is the heat kernel of \eqref{eq:LB_on_S^7} with respect to the measure $\sin^6\eta \, {\rm d}\eta$, $\eta\in [0,\pi)$, and $p^{\widetilde{\square}_{{\rm AdS}^{15}(\mathbb{O})}}_t(r,u)$ the heat kernel at $(0,0)$ of $\widetilde{\square}_{{\rm AdS}^{15}(\mathbb{O})}$ with respect to the measure in~\eqref{eq:measure}, i.e.,
\begin{equation*}
{\rm d}\mu(r,u)=\frac{{\pi}^7}{90}\sinh^7 r\cosh^7 r\sin^6 u \,{\rm d}r\, {\rm d}u, \qquad r\in [0,\infty),\qquad u\in[0,\pi].
\end{equation*}
In order to write \eqref{eq:9} more precisely, let us consider the analytic change of variables $\tau\colon (r,\eta) \rightarrow (r,{\rm i}\eta)$ that will be applied on functions of the type $f(r,\eta)=h(r){\rm e}^{-{\rm i}\lambda \eta}$, with $h$ smooth and compactly supported on $[0,\infty)$ and $\lambda>0$. Then as we saw in the proof of Proposition~\ref{prop:1} one can see that
\begin{equation*}
\widetilde{\square}_{{\rm AdS}^{15}(\mathbb{O})}(f\circ \tau)=\big(\widetilde{\triangle}_{\mathbb{H}^{15}}f\big)\circ \tau,
\end{equation*}
where
\begin{equation*}
\widetilde{\triangle}_{\mathbb{H}^{15}}=\widetilde{\triangle}_{\mathbb{O}H^1}+\frac{1}{\cosh^2 r}\widetilde{\triangle}_{P},\qquad \widetilde{\triangle}_{P}=\frac{{\partial}^2}{\partial {\eta}^2}+6\coth \eta\frac{{\partial}}{\partial {\eta}}.
\end{equation*}

Then, one deduces
\begin{equation*}
{\rm e}^{t\widetilde{L}}(f\circ \tau)={\rm e}^{t\widetilde{\triangle}_{\mathbb{S}^7}}{\rm e}^{t\widetilde{\square}_{{\rm AdS}^{15}(\mathbb{O})}}(f \circ \tau)={\rm e}^{t\widetilde{\triangle}_{\mathbb{S}^7}}\big(\big({\rm e}^{t{\widetilde{\triangle}_{\mathbb{H}^{15}}}}f\big) \circ \tau\big)=\big({\rm e}^{-t\widetilde{\triangle}_{P}}{\rm e}^{t{\widetilde{\triangle}_{\mathbb{H}^{15}}}}f\big) \circ \tau.
\end{equation*}
Now, since for every $f(r,\eta)=h(r){\rm e}^{-{\rm i}\lambda \eta}$,
\begin{equation*}
\big({\rm e}^{t\widetilde{\square}_{{\rm AdS}^{15}(\mathbb{O})}}f\big)(0,0)=\big({\rm e}^{t{\widetilde{\triangle}_{\mathbb{H}^{15}}}}\big)\big(f\circ \tau^{-1}\big)(0,0),
\end{equation*}
one deduces that for a function $h$ depending only on~$u$,
\begin{equation*}
\int_{0}^{\pi}h(u)p^{\widetilde{\square}_{{\rm AdS}^{15}(\mathbb{O})}}_t(r,u)\sin^6u\,{\rm d}u=\int_{0}^{\infty}h(-{\rm i}u)q_{t,15}(\cosh r \cosh u) \sinh^6u\,{\rm d}u.
\end{equation*}
Therefore, coming back to \eqref{eq:9}, one infers that using the analytic extension of~$s_t$ one must have
\begin{gather*}
\int_{0}^{\pi}s_t(\eta,u)p^{\square_{{\rm AdS}^{15}(\mathbb{O})}}_t(r,u)\sin^6u\,{\rm d}u=\int_{0}^{\infty}s_t(\eta,-{\rm i}u)q_{t,15}(\cosh r \cosh u)\sinh^6u\,{\rm d}u,
\end{gather*}
where $q_{t,15}$ is the Riemannian heat kernel on the real hyperbolic space $\mathbb{H}^{15}$ given in~\eqref{eq:riem.heat.kernel.on.H^15}.
\end{proof}

\subsection{Second integral representation}\label{section4.2}

\begin{Proposition}\label{prop:3}
For $r\geq 0$, $\eta\in [0,\pi]$, and $t>0$ we have
\begin{gather*}
p_t(r,\eta)= \int_0^\pi \int_{0}^{\infty} G_t(\eta, \varphi, u) q_{t,9}(\cosh r \cosh u) \sin^5 \varphi \, {\rm d}u \, {\rm d} \varphi .
\end{gather*}
where $q_{t,9}$ is the $9$-dimensional Riemannian heat kernel on the hyperbolic space $\mathbb{H}^{9}$:
\begin{gather*}
q_{t,9}(\cosh s):=\frac{{\rm e}^{-16t}}{(2\pi)^4 \sqrt{4\pi t}}\left(\frac{1}{\sinh s}\frac{{\rm d}}{{\rm d}s}\right)^{4}{\rm e}^{-s^2/4t},
\end{gather*}
and
\begin{gather}\label{function G}
 G_t(\eta, \varphi, u)
= \frac{15}{8} \sum_{m\geq 0}{\rm e}^{-(m(m+6)+33)t} (\cos \eta + {\rm i}\sin \eta \cos \varphi)^m \cosh((m+3)u).
\end{gather}
\end{Proposition}
\begin{proof}
The strategy of the following method appeals to some results proved in~\cite{AIMVOM97}. Firstly, we decompose the subelliptic heat kernel in the~$\eta$ variable with respect to the basis of normalized eigenfunctions of $\tilde{\triangle}_{S^7}=\frac{{\partial}^2}{\partial {\eta}^2}+6\cot \eta\frac{{\partial}}{\partial {\eta}}$.
Accordingly,
\begin{equation*}
p_t(r,\eta)=\sum_{m\geq 0} f_m(t,r)h_m(\eta),
\end{equation*}
where for each $m$, $h_m$ is given by
\begin{equation*}
h_m(\eta)=\frac{15}{16}\int_{0}^{\pi}{(\cos \eta + {\rm i}\sin \eta \cos \varphi)}^m{\sin^5 \varphi}\, {\rm d} \varphi
\end{equation*}
and $f_m(t,\cdot)$ solves the following heat equation
\begin{align*}
\frac{\partial}{\partial t}f_m(t,r)& =\left( \frac{{\partial}^2}{\partial {r}^2}+(7\coth r+7\tanh r)\frac{{\partial}}{\partial {r}}-m(m+6){\tanh^2 r} \right) f_m(t,r) \\
&=\left( \frac{{\partial}^2}{\partial {r}^2}+(7\coth r+7\tanh r)\frac{{\partial}}{\partial {r}}+\frac{m(m+6)}{\cosh^2 r}-m(m+6) \right) f_m(t,r).
\end{align*}
We consider then the operator
\begin{equation*}
L_m:=\frac{{\partial}^2}{\partial {r}^2}+(7\coth r+7\tanh r)\frac{{\partial}}{\partial {r}}+\frac{m(m+6)}{\cosh^2 r}+49,
\end{equation*}
which was studied in \cite[p.~229]{AIMVOM97}.
From \cite[Theorem~2]{AIMVOM97}, with $\alpha=3+\frac{m}{2}, \beta=-\frac{m}{2}$, we deduce that the solution to the wave Cauchy problem associated with the subelliptic Laplacian is given $f\in C^{\infty}_{0}\big(\mathbb{O}H^1\big)$ by
\begin{equation*}
\cos \big(s\sqrt{-L_{m}}\big)(f)(w)=\frac{-\sinh s}{(2\pi)^{4}}\left(\frac{1}{\sinh s }\frac{{\rm d}}{{\rm d}s}\right)^{4}\int_{\mathbb{O}H^1}K_m(s,w,y)f(y)\frac{{\rm d}y}{\big(1-||y||^2\big)^{8}},
\end{equation*}
where
\begin{gather*}
K_m(s,w,y) =\frac{(1-\overline{\langle w,y\rangle})^{3+m/2}}{(1-\langle w,y\rangle)^{m/2}}\frac{1}{\cosh^3(d(w,y))\sqrt{\cosh^2(s)-\cosh^2(d(w,y))}} \\
\hphantom{K_m(s,w,y) =}{} \times {_2}F_{1}\left(m+3,-m-3,\frac{1}{2};\frac{\cosh(d(w,y))-\cosh(s)}{2\cosh(d(w,y)) }\right),
\end{gather*}
where $_2 F_{1}$ is the Gauss hypergeometric function and ${\rm d}y$ stands for the Lebesgue measure in~$\mathbb{R}^8$. Using the spectral formula
\begin{equation*}
{\rm e}^{tL}=\frac{1}{\sqrt{4\pi t}}\int_{\mathbb{R}}{\rm e}^{-s^2/(4t)}\cos\big(s\sqrt{-L}\big)\,{\rm d}s,
\end{equation*}
which holds for any non positive self-adjoint operator, we deduce that the solution to the heat Cauchy problem associated with~$L_m$:
\begin{gather*}
{\rm e}^{t L_m}(f)(w) =\frac{{\rm e}^{-m(m+6)t-7^2 t}}{\sqrt{4\pi t}(2\pi)^{4}}
\int_{\mathbb{R}} {\rm d}s (-\sinh s){\rm e}^{-s^2/(4t)}\\
\hphantom{{\rm e}^{t L_m}(f)(w) =}{}\times \left(\frac{1}{\sinh s}\frac{{\rm d}}{{\rm d}s}\right)^{4} \int_{\mathbb{O}H^1} K_m(s,w,y)f(y)\frac{{\rm d}y}{\big(1-||y||^2\big)^{8}}.
\end{gather*}

Performing integration by parts $4$-times,
\begin{gather*}
\int_{\mathbb{R}} {\rm d}s (-\sinh s)\left(\frac{1}{\sinh s}\frac{{\rm d}}{{\rm d}s}\right)^{4} {\rm e}^{-s^2/(4t)} \int_{\mathbb{O}H^1} K_m(s,w,y)f(y)\frac{{\rm d}y}{\big(1-||y||^2\big)^{8}}
\\
\qquad{}=\int_{\mathbb{O}H^1}f(y)\frac{{\rm d}y}{(1-||y||^2)^{8}} \int_{\mathbb{R}} {\rm d}s (-\sinh s)K_m(s,w,y)\left(\frac{1}{\sinh s }\frac{{\rm d}}{{\rm d}s}\right)^{4}{\rm e}^{-s^2/4t}
\\
\qquad{} =2\int_{\mathbb{O}H^1}f(y)\frac{{\rm d}y}{\big(1-||y||^2\big)^{8}} \int_{d(w,y)}^{\infty}{\rm d}(\cosh(s))K_m(s,w,y)\left(\frac{1}{\sinh s }\frac{{\rm d}}{{\rm d}s}\right)^{4}{\rm e}^{-s^2/4t}.
\end{gather*}
Thus we get
\begin{equation*}
{\rm e}^{tL_m}(f)(0)=2 {\rm e}^{-(m(m+6)+33)t}\int_{\mathbb{O}H^1} f(y)\frac{{\rm d}y}{\big(1-||y||^2\big)^{8}}\int_{d(0,y)}^{\infty}{\rm d}(\cosh s)K_m(s,0,y)q_{t,9}(\cosh s).
\end{equation*}
As a result, the subelliptic heat kernel of $L_m$ reads
\begin{gather*}
\frac{{\rm d}y}{\big(1-||y||^2\big)^{8}}\int_{d(0,y)}^{\infty}{\rm d}(\cosh s)K_m(s,0,y)q_{t,9}(\cosh s)
\\
\qquad{} ={\rm d}r \sinh^{7}r\cosh^7 r\int_{r}^{\infty}{\rm d}(\cosh s)K_m(s,0,y)q_{t,9}(\cosh s).
\end{gather*}
By changing the variable $\cosh s=\cosh r \cosh u$ for $u\geq 0$, the last expression becomes
\begin{equation*}
{\rm d}r \sinh^{7}r\cosh^7 r\int_{0}^{\infty} {_2} F_{1}\left(m+3,-m-3,\frac{1}{2};\frac{1-\cosh u }{2 }\right) q_{t,9}(\cosh r\cosh u)\,{\rm d}u.
\end{equation*}
Therefore $p_t(r,\eta)$ has the integral representation
\begin{equation*}
2\sum_{m\geq 0}{\rm e}^{-(m(m+6)+33)t}h_m(\eta) \int_{0}^{\infty} {_2} F_{1}\left(m+3,-m-3,\frac{1}{2};\frac{1-\cosh u }{2 }\right) q_{t,9}(\cosh r \cosh u)\,{\rm d}u.
\end{equation*}

Now, notice that $_2 F_{1} \big(m+3,-m-3,\frac{1}{2};\frac{1-\cosh u }{2 }\big)$ is simply the Cheybyshev polynomial of the first kind
\begin{equation*}
T_{m+3}(x)={_2}F_{1} \left(m+3,-m-3,\frac{1}{2};\frac{1-x }{2 }\right),
\end{equation*}
for all $x\in \mathbb{C}$. Therefore, one has
\begin{equation*}
_2 F_{1} \left(m+3,-m-3,\frac{1}{2};\frac{1-\cosh u }{2 }\right)=T_{m+3}(\cosh u)=\cosh((m+3)u),
\end{equation*}
and the proof is over.
\end{proof}

\subsection*{Acknowledgements}
F.B.~is partially funded by the NSF grant DMS-1901315.

\pdfbookmark[1]{References}{ref}
\LastPageEnding


\begin{thebibliography}{99}
\footnotesize\itemsep=0pt

\bibitem{FBGC19}
Baudoin F., Cho G., The subelliptic heat kernel of the octonionic {H}opf
 fibration, \href{https://doi.org/10.1007/s11118-020-09854-4}{\textit{Potential Anal.}}, {t}o appear, \href{https://arxiv.org/abs/1904.08568}{arXiv:1904.08568}.

\bibitem{FBND18}
Baudoin F., Demni N., Integral representation of the sub-elliptic heat kernel
 on the complex anti-de {S}itter fibration, \href{https://doi.org/10.1007/s00013-018-1201-1}{\textit{Arch. Math. (Basel)}}
 \textbf{111} (2018), 399--406, \href{https://arxiv.org/abs/1802.04199}{arXiv:1802.04199}.

\bibitem{FBJW18}
Baudoin F., Demni N., Wang J., The horizontal heat kernel on the quaternionic
 anti-de {S}itter spaces and related twistor spaces, \href{https://doi.org/10.1007/s11118-018-9746-y}{\textit{Potential Anal.}}
 \textbf{52} (2020), 281--300, \href{https://arxiv.org/abs/1805.06796}{arXiv:1805.06796}.

\bibitem{BGMR}
Baudoin F., Grong E., Molino G., Rizzi L., $H$-type foliations,
 \href{https://arxiv.org/abs/1812.02563}{arXiv:1812.02563}.

\bibitem{MR3719061}
Baudoin F., Wang J., Stochastic areas, winding numbers and {H}opf fibrations,
 \href{https://doi.org/10.1007/s00440-016-0745-x}{\textit{Probab. Theory Related Fields}} \textbf{169} (2017), 977--1005,
 \href{https://arxiv.org/abs/1602.06470}{arXiv:1602.06470}.

\bibitem{MR1877586}
B\u{a}di\c{t}oiu G., Ianu\c{s} S., Semi-{R}iemannian submersions from real and
 complex pseudo-hyperbolic spaces, \href{https://doi.org/10.1016/S0926-2245(01)00070-5}{\textit{Differential Geom. Appl.}}
 \textbf{16} (2002), 79--94, \href{https://arxiv.org/abs/math.DG/0005228}{arXiv:math.DG/0005228}.

\bibitem{MR2541147}
Eldredge N., Precise estimates for the subelliptic heat kernel on {$H$}-type
 groups, \href{https://doi.org/10.1016/j.matpur.2009.04.011}{\textit{J.~Math. Pures Appl.}} \textbf{92} (2009), 52--85,
 \href{https://arxiv.org/abs/0810.3218}{arXiv:0810.3218}.

\bibitem{AIMVOM97}
Intissar A., Ould~Moustapha M.V., Explicit formulae for the wave kernels for
 the {L}aplacians {$\Delta_{\alpha\beta}$} in the {B}ergman ball {$B^n$},
 {$n\geq 1$}, \href{https://doi.org/10.1023/A:1006501627929}{\textit{Ann. Global Anal. Geom.}} \textbf{15} (1997), 221--234.

\bibitem{MR1803808}
Tian Y., Matrix representations of octonions and their applications,
 \href{https://doi.org/10.1007/BF03042010}{\textit{Adv. Appl. Clifford Algebras}} \textbf{10} (2000), 61--90,
 \href{https://arxiv.org/abs/math.RA/0003166}{arXiv:math.RA/0003166}.

\bibitem{JW16}
Wang J., The subelliptic heat kernel on the anti-de {S}itter space,
 \href{https://doi.org/10.1007/s11118-016-9561-2}{\textit{Potential Anal.}} \textbf{45} (2016), 635--653, \href{https://arxiv.org/abs/1204.3642}{arXiv:1204.3642}.

\end{thebibliography}
\end{document}